\documentclass[11pt]{amsart}
\usepackage{amssymb,amsmath, amsfonts, enumerate, url, xy, latexsym, url}
\xyoption{all}

\newtheorem{thm}{Theorem}[section]

\newtheorem{lem}[thm]{Lemma}

\theoremstyle{definition}

\newtheorem{remark}[thm]{Remark}

\DeclareMathAlphabet{\mbf}{OML}{cmm}{b}{it}

\newcommand\Ind{\operatorname{Ind}}
\newcommand\Core{\operatorname{Core}}

\newcommand{\UD}{\operatorname{UD}}
\newcommand\ed{\operatorname{ed}}

\newcommand\Mat{\operatorname{M}}

\newcommand\Gal{\operatorname{Gal}}
\newcommand\rank{\operatorname{rank}}

\newcommand\PGL{\operatorname{PGL}}
\newcommand\PGLn{\operatorname{PGL}_n}

\newcommand\bbZ{\mathbb{Z}}

\newcommand\bbQ{\mathbb{Q}}

\begin{document}

\title[Essential dimension]{An upper bound on the essential dimension of a
central simple algebra}
\author[Aurel Meyer]{Aurel Meyer$^\dagger$}
\thanks{$^\dagger$ Aurel Meyer was partially supported by
a University Graduate Fellowship at the University of British Columbia}
\author[Zinovy Reichstein]{Zinovy Reichstein$^{\dagger \dagger}$}
\thanks{$^{\dagger \dagger}$ Z. Reichstein was partially supported by
NSERC Discovery and Accelerator Supplement grants}

\address{Department of Mathematics, University of British Columbia,
Vancouver, BC V6T 1Z2, Canada}



\begin{abstract} 
We prove a new upper bound on the essential $p$-dimension of the projective
linear group $\PGL_n$.
\end{abstract}

\maketitle
\tableofcontents

\section{Introduction}

Let $k$ be a base field; all other fields will be assumed 
to be extensions of $k$.

Given a central simple algebra $A$ over a field $K$ one 
can ask whether $A$ can be written as $A=A_0\otimes_{K_0} K$ 
where $A_0$ is a central simple algebra over some subfield $K_0$ of $K$.
In that situation we say that $A$ {\em descends} to $K_0$.
The {\em essential dimension} of $A$, denoted $\ed(A)$, 
is the minimal transcendence degree over $k$ of 
a field $K_0\subset K$ such that $A$ descends to $K_0$.
It can be thought of as ``the minimal number of independent 
parameters" required to define $A$.

For a prime number $p$, the related notion of essential dimension at $p$ 
of an algebra $A/K$ is defined as $\ed(A;p)=\min\, \ed(A_{K^\prime})$,
where $K^\prime/K$ runs over all finite field extensions 
of degree prime to $p$.

We also define
\[ \ed(\PGL_n) := \text{\bf max} \, \{ \, \ed(A) \;  \} \, , \]
and
\[ \ed(\PGL_n; p) := \text{\bf max} \, \{ \, \ed(A; p) \;  \} \, , \]
where the maximum is taken over all fields $K/k$ and
over all central simple $K$-algebras $A$ of degree $n$.
The appearance of $\PGL_n$ in the symbols $\ed(\PGL_n)$ 
and $\ed(\PGL_n; p)$ has to do with the fact that central simple 
algebras of degree $n$ are in a natural bijective correspondence 
with $\PGL_n$-torsors. In fact, one 
can define $\ed(G)$ and $\ed(G; p)$ 
for every algebraic $k$-group $G$ in a similar manner, 
using $G$-torsors instead of central simple algebras; 
see~\cite{reichstein2}, \cite{ry} or \cite{bf}.

To the best of our knowledge, 
the problem of computing $\ed(\PGL_n)$ was first
raised by C.~Procesi in the 1960s.  Procesi and S. Amitsur constructed
so-called {\em universal division algebras} $\UD(n)$
and showed that $\UD(n)$ has various
generic properties among central simple algebras of
degree $n$. In particular, their arguments can be used to
show that 
\[ \text{$\ed(\UD(n)) \ge \ed(A)$
and
$\ed(\UD(n); p) \ge \ed(A; p)$} \]
for any prime integer $p$; cf.~\cite[Remark 2.8]{lrrs}.  
Equivalently,
\[ \text{$\ed (\UD(n)) = \ed(\PGL_n)$ and
$\ed (\UD(n); p) = \ed(\PGL_n; p)$.} \]
Since the center of $\UD(n)$ has transcendence degree $n^2 +1$
over $k$, we conclude that $\ed(\PGL_n) \le n^2 + 1$.
Procesi showed (using different terminology) that in fact,
\[ \ed(\PGL_n) \leq n^2 \, ; \]
see~\cite[Theorem 2.1]{procesi}.

The problem of computing $\ed(\PGL_n)$ was raised again 
by B. Kahn in the early 1990s. In particular, in 1992
Kahn asked the second author if $\ed(\PGL_n)$ 
grows sublinearly in $n$, i.e., whether
\[ \ed(\PGL_n) \le a n + b \] for some positive 
real numbers $a$ and $b$.  To the best of our 
knowledge, this question never appeared in print but 
it is implicit in~\cite[Section 2]{kahn}. It remains open;
the best known upper bound,
\begin{equation} \label{e.upper-bound}
\ed(\PGL_n) \le \begin{cases}
\text{$\frac{(n-1)(n-2)}{2}$, for every odd $n \ge 5$ and} \\
\text{$n^2 - 3n + 1$, for every $n \ge 4$}
\end{cases}
\end{equation}
(see~\cite{lr}, \cite[Theorem 1.1]{lrrs},~\cite[Proposition 1.6]{lemire}
and~\cite{ff}), is quadratic in $n$ and the best known lower bound, 
\[ \ed(\PGL_{p^r})\ge\ed(\PGL_{p^r}; p) \ge 2r \, , \]
is logarithmic.


Note that if $p^s$ is the largest power of $p$ dividing $n$ then
one easily checks, using primary decomposition of central simple
algebras, that $\ed(\PGL_n; p) = \ed(\PGL_{p^s}; p)$.
Thus for the purpose of computing $\ed(\PGL_n; p)$ it suffices to consider
the case where $n = p^s$. In this case we have showed that 
\[ \ed(\PGL_{p^s}; p) \le p^{2s-1}-p^s+1 \]
for any $s \ge 2$; see~\cite[Corollary 1.2]{mr}.
The main result of this paper is the following stronger upper bound.

\begin{thm}\label{thm.pgl}
Let $n=p^s$ for some $s\ge 2$. 
Then
\[
\ed(\PGL_n; p)\le 2\frac{n^2}{p^2}-n+1
\]
\end{thm}

A. S. Merkurjev~\cite{me2} recently showed that for $s = 2$ this 
bound is sharp, i.e., $\ed(\PGL_{p^2};p)=p^2 + 1$. We conjecture
that this bound is sharp for every $s \ge 2$; this would imply,
in particular, that $\ed(\PGLn)$ is not sublinear in $n$.

Our upper bound on $\ed(\PGL_n; p)$ is a consequence of the 
following result. Here $n$ is not assumed to be a prime power.

\begin{thm}\label{thm.csa}
Let $A/K$ be a central simple algebra of degree $n$.
Suppose $A$ contains a field $F$, Galois over $K$
and $\Gal(F/K)$ can be generated by $r \ge 1$ elements.
If $[F: K] = n$ then we further assume that $r \ge 2$.
Then
\[
\ed(A)\le r\frac{n^2}{[F:K]}-n+1
\]
\end{thm}

Note that we always have $[F: K] \le n$. In the special 
case where equality holds, i.e., $A$ is a crossed product
in the usual sense, Theorem~\ref{thm.csa} reduces 
to~\cite[Corollary 3.10(a)]{lrrs}. 

To deduce Theorem~\ref{thm.pgl} from Theorem~\ref{thm.csa},
let $n = p^s$ and $A = \UD(n)$.  In~\cite[1.2]{rs}, 
L. H. Rowen and D. J. Saltman 
showed that if $s \ge 2$ then there is a finite field 
extension $K^\prime/K$ of degree prime to $p$,
such that $A^\prime := A\otimes_KK^\prime$ 
contains a field $F$, Galois over $K^\prime$ 
with $\Gal(F/K') \simeq \bbZ/p\times \bbZ/p$.  
Thus, if $s \ge 2$, Theorem~\ref{thm.csa} tells us that
\[
\ed(\PGL_n; p)=\ed(A;p)\le \ed(A^\prime) \le 2\frac{n^2}{p^2}-n+1 \, . 
\]
This proves Theorem~\ref{thm.pgl}.
\qed

The remainder of this paper will be devoted to proving Theorem~\ref{thm.csa}.
We reduce the problem to a question about $G$-lattices, using the same 
approach as in~\cite[Sections 2--3]{lrrs}, but our analysis is more 
delicate here, and the results (Theorems~\ref{thm.csa} and~\ref{thm.H}) 
are stronger.

\section{$G/H$-crossed products}

\begin{lem} \label{lem2.1}
In the course of proving Theorem~\ref{thm.csa} we may assume 
without loss of generality that $F$ is contained in a subfield
$L$ of $A$ such that $L/K$ is a separable extension of degree 
$n = \deg(A)$.
\end{lem}

\begin{proof} Note that we are free to replace $K$ by $K(t)$,
$F$ by $F(t)$ and $A$ by $A(t) = A \otimes_K K(t)$, where $t$ 
is an independent variable.
Indeed, $\ed_k \, A(t) = \ed_k(A)$; see, 
e.g.,~\cite[Lemma 2.7(a)]{lrrs}.
Thus if the inequality of Theorem~\ref{thm.csa} is 
proved for $A(t)$, it will also hold for $A$.  

The advantage of passing from $A$ to $A(t)$ is that $K(t)$ if Hilbertian 
for any infinite field $K$; see, e.g.,~\cite[Proposition 13.2.1]{fj}.
Thus after adjoining two variables, $t_1$ and $t_2$ as above, 
we may assume without loss of generality that $K$ is Hilbertian. 
(Note that a subfield $L \subset A$ of degree $n$ over $K$
may not exist without this assumption.)

Let $F \subset F'$ be maximal among separable field extension 
of $F$ contained in $A$. We will look for $L$ inside the centralizer 
$C_A(F')$. By the Double Centralizer Theorem, 
$C_A(F')$ is a central simple algebra with center $F'$. The maximality 
of $F'$ tells us that $C_A(F')$ contains no non-trivial field 
extensions of $F'$. In particular, $C_A(F') = \Mat_r(F')$, where 
$r [F':K] = n$. 

On the other hand, since $K$ is Hilbertian, so is 
its finite separable extension $F'$; cf.~\cite[12.2.3]{fj}.
Consequently, $F'$ admits a finite separable extension 
$L/F'$ of degree $r$. (To construct $L/F'$, 
start with the field extension $L_r = F'(t_1, \dots, t_r)[x]/(f(x))$ 
of $F'(t_1, \dots, t_n)$
of degree $r$, where $f(x) = x^r + t_1 x^{r-1} + \ldots + t_{n-1} x + t_n$
is the general polynomial of degree $r$. Then specialize $t_1, \dots, t_r$
in $F'$, using the Hilbertian property, to obtain a field extension
$L/F'$ of degree $r$.) Any such $L/F'$ can be embedded 
into $\Mat_r(F')$ via the regular representation 
of $L$ on $L = (F')^r$; cf.~\cite[Lemma 13.1a]{pierce}. 
By the maximality of $F'$, we conclude that 
$L = F'$, i.e., $r = 1$ and $[L:K] = n$, as desired.
\end{proof} 

 Let us now assume that our central simple algebra $A/K$ 
has a separable maximal subfield
$L/K$, as in Lemma~\ref{lem2.1}. We will denote the Galois closure of
$L$ over $K$ by $E$ and the associated Galois groups by $G= \Gal(E/K)$, 
$H= \Gal(E/L)$ and $N=\Gal(E/F)$, as in the diagram below.
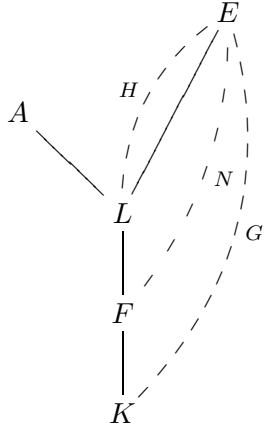
\[
\xymatrix{
&&E\ar@{-}[ddl]\ar@/^-1pc/@{--}[ddl]_H\ar@/^1pc/@{--}[dddl]^N\ar@/^2pc/@{--}[ddddl]^G\\
A\ar@{-}[dr]&&\\
&L\ar@{-}[d]&\\
&F\ar@{-}[d]&\\
&K&
}
\]
In the terminology of~\cite{lrrs}, $A/K$ is a $G/H$-crossed product;
cf. also~\cite[Appendix]{fss}. Note that since $E/K$
is the smallest Galois extension containing $L/K$, we have 
\begin{equation} \label{e.H}
\Core_G(H)=\bigcap_{g \in G} \, H^g = \{ 1 \} \, .
\end{equation}
where $H^g := g H g^{-1}$. We will assume that this condition is
satisfied whenever we talk about $G/H$-crossed products.

Using the notation introduced above and remembering that
$[G: H] = [L: K] = \deg(A) = n$, and $\dfrac{n}{[F: K]} = [L : F] = [N : H]$, 
we can restate Theorem~\ref{thm.csa} as follows.

\begin{thm} \label{thm.csa'}
Let $A$ be a $G/H$-crossed product. Suppose $H$ is contained 
in a normal subgroup $N$ of $G$ and $G/N$ is generated by $r$ elements.
Furthermore, assume that either $H \ne \{ 1 \}$ or $r \ge 2$. Then
\[ \ed(A) \le r [G: H] \cdot [N: H] - [G: H] + 1 \, . \]
\end{thm}

\section{$G$-lattices}

In the sequel $H \le G$ will be finite groups. Given $g \in G$
we will write $\overline{g}$ for the left coset $gH$ of $H$.
We will denote the identity element of $G$ by $1$.

Recall that a $G$-lattice $M$ is a (left) $\bbZ[G]$-module, 
which is free of finite rank over $\bbZ$. In particular, any 
finite set $X$ with a $G$-action gives rise to a $G$-lattice 
$\bbZ[X]$; $G$-lattices of this form are called {\em permutation}.
For background material on $G$-lattices we refer the reader 
to~\cite{lorenz}.

Of particular interest to us will be the $G$-lattice
$\omega(G/H)$, which is defined as the kernel of 
the natural augmentation map $\bbZ[G/H] \to \bbZ$, sending
$n_1 \overline{g_1} + \dots + n_s \overline{g_s}$ to 
$n_1 + \dots + n_s$. 

The starting point for our proof 
of Theorem~\ref{thm.csa'} (and hence, of Theorem~\ref{thm.csa})
will be the following result from~\cite{lrrs}.

\begin{thm} \label{thm.lrrs} {\em(}\cite[Theorem 3.5]{lrrs}{\em)}
Let $P$ be a permutation $G$-lattice and
\[ 0 \to M \to P \to \omega(G/H) \to 0 \]
be an exact sequence of $G$-lattices.
If the $G$-action on $M$ is faithful then
\[ \ed(A) \le \rank(M) - n + 1 \]
for any $G/H$-crossed product $A$.
\qed
\end{thm}

The condition that $G$ acts faithfully on $M$ is not automatic.
However, the following lemma shows that it is satisfied for many 
natural choices of $P$.

\begin{lem} \label{lem2.1'}
Let $G \ne \{ 1 \}$ be a finite group $H \le G$ be a subgroup of $G$, 
$H_1, \ldots, H_r$ be subgroups of $H$ and
\begin{equation} \label{e2.1'}
0 \to M \to \oplus_{i = 1}^r \bbZ [G/H_i] \to \omega(G/H) \to 0 
\end{equation}
be an exact sequence of $G$-lattices.  Assume that
$H$ does not contain any nontrivial normal subgroup of $G$
(i.e., $H$ satisfies condition~\eqref{e.H} above).
Then the $G$-action on $M$ fails to be faithful 
if and only if $s = 1$ and $H_1 = H$.
\end{lem}

Here we are not specifying the map $\oplus_{i = 1}^r \bbZ [G/H_i] 
\to \omega(G/H)$; the lemma holds for any exact sequence of 
the form~\eqref{e2.1'}. We also note that in the case where
$H_1 = \dots = H_r = \{ 1 \}$, Lemma~\ref{lem2.1'} reduces 
to~\cite[Lemma 2.1]{lrrs}.

\begin{proof} To determine whether or not the $G$-action 
on $M$ is faithful, we may replace $M$ by $M_{\bbQ} := M \otimes \bbQ$. 
After tensoring with $\bbQ$, the sequence~\eqref{e2.1'} splits, and we
have an isomorphism
\begin{equation} \label{e.Q-rep} 
\omega(G/H)_{\bbQ} \oplus M_{\bbQ} \simeq \oplus_{i = 1}^r \bbQ [G/H_i] \, . 
\end{equation}

Case 1: $r \ge 2$. Then $H_r$ is a subgroup of $H$, we have a
natural surjective map $\bbQ [G/H_r] \to \bbQ [G/H]$. Using complete 
irreducibility over $\bbQ$ once again, we see that
$\bbQ [G/H]$ (and hence $\omega(G/H)$) is
a subrepresentation of $\bbQ [G/H_r]$.  Thus~\eqref{e.Q-rep} tells us that 
$\bbQ [G/H_{r-1}]$ is a subrepresentation of $M_{\bbQ}$.
The kernel of the $G$-representation 
on $\bbQ [G/H_{r-1}]$ is a normal subgroup of $G$ contained in 
$H_{r-1}$ (and hence, in $H$); by our assumption on $H$, any such 
subgroup is trivial. This shows that $G$ acts faithfully on  
$\bbQ [G/H_{r-1}]$ and hence, on $M$.

\smallskip
Case 2: Now assume $r = 1$. Our exact sequence now assumes the form
\[ 0 \to M_{\bbQ} \to \bbQ [G/H_1] \to \omega(G/H)_{\bbQ} \to 0 \, . \]
If $H = H_1$ then $M \simeq \bbZ$, with trivial (and hence, non-faithful)
$G$-action.

Our goal is thus to show that if $H_1 \subsetneq H$ then
the $G$-action on $M_{\bbQ}$ is faithful. 
Denote by $\bbQ[1]$ the trivial representation (it will be clear from the context of which group).
Observe that
\[ \begin{array}{lcl}
\bbQ [G/H_1] & \simeq & \Ind_{H_1}^G \bbQ[1] \,\simeq\,  \Ind_H^G\Ind_{H_1}^H \bbQ[1] \,\simeq\, \Ind_H^G \bbQ[H/H_1]\\
	& \simeq & \Ind_H^G (\omega(H/H_1)_{\bbQ}\oplus\bbQ[1])\\
	& \simeq & \Ind_H^G \omega(H/H_1)_{\bbQ} \,\oplus\,\bbQ[G/H]\\
	& \simeq & \Ind_H^G \omega(H/H_1)_{\bbQ} \,\oplus\,\omega(G/H)_{\bbQ}\,\oplus\, \bbQ[1]  
\end{array} \]

and we obtain
\[ M_{\bbQ} \simeq \Ind_H^G \omega(H/H_1)_{\bbQ} \,\oplus\,\bbQ[1]\, . \]
If $H_1 \subsetneq H$ 
then the kernel of the $G$-representation $\Ind_H^G \omega(H/H_1)_{\bbQ}$ 
is a normal subgroup of $G$ contained in $H_1$ (and hence, in $H$).
By our assumption on $H$, this kernel is trivial.
\end{proof}

\section{An upper bound}

In this section we will prove the following upper bound
on the essential dimension of a $G/H$-crossed product.

We will say that $g_1, \dots, g_s \in G$ generate $G$ over $H$ if
$G = \langle g_1, \dots, g_s, H \rangle$. 

\begin{thm} \label{thm.H} Let $A$ be a $G/H$-crossed product.
Suppose that

\smallskip
(i) $g_1, \dots, g_s \in G$ generate $G$ over $H$, and 

\smallskip
(ii) if $G$ is cyclic then $H \ne \{ 1 \}$.

\smallskip
\noindent
Then $\ed(A) \le \sum_{i = 1}^s [G: (H \cap H^{g_i})]  - [G:H] + 1$.
\end{thm}

\begin{remark} \label{rem.H} The index
$[G: (H \cap H^{g_i})]$ appearing in the above formula
can be rewritten as 
\[ [G: H] \cdot [H: (H \cap H^{g_i})] = [G : H] \cdot 
[(H \cdot H^{g_i}): H] \, ; \]
see, e.g., \cite[1.3.11(i)]{robinson}.  Note 
$H \cdot H^g := \{ hh' \, | \, h \in H, \; h' \in H^g \}$ 
is a subset of $G$ but may not be a subgroup, and
$[(H \cdot H^g): H]$ is defined as $\dfrac{|H \cdot H^g|}{|H|}$.

If $H$ is contained in a normal subgroup $N$ of $G$ then clearly
$H \cdot H^g$ lies in $N$, each $[H \cdot H^g: H] \le [N: H]$
and thus Theorem~\ref{thm.H} yields
\[ \ed(A) \le s [G:H] \cdot [N: H] -[G:H]+1\, . \] 
This is a bit weaker than the inequality of Theorem~\ref{thm.csa'},
even though the two look very similar. The difference 
is that we have replaced $r$ in the inequality of Theorem~\ref{thm.csa'}
by $s$, where $G$ is generated by $s$ elements over $H$ and 
by $r$ elements over $N$. A priori $r$ can
be smaller than $s$. Nevertheless in the next section
we will deduce Theorem~\ref{thm.csa'}
from Theorem~\ref{thm.H} by a more delicate argument 
along these lines.
\end{remark}

Our proof of Theorem~\ref{thm.H} will rely on the following lemma.

\begin{lem} \label{lem.subgroup}
Let $V$ be a $\bbZ[G]$-submodule of $\omega(G/H)$. Then
\[  G_V := \{ g \in G \, | \,  \overline{g} - \overline{1} \in V \} \]
is a subgroup of $G$ containing $H$.
\end{lem}

\begin{proof} The inclusion $H \subset G_V$ is obvious from the definition.

To see that $G_V$ is closed under multiplication, suppose
$g, g' \in G_V$. That is, both $\overline{g} - \overline{1}$ and 
$\overline{g'} - \overline{1}$ lie in $V$. Then
\[ \overline{g g'} - \overline{1} =  
g \cdot (\overline{g'} - \overline{1}) +  
(\overline{g} - \overline{1}) \] 
also lies in $V$, i.e., $g g ' \in G_V$, as desired.
\end{proof}

\begin{proof}[Proof of Theorem~\ref{thm.H}]
We claim that the elements 
$\overline{g_1} - \overline{1}, \ldots, \overline{g_s} - \overline{1}$
generate $\omega(G/H)$ as a $\bbZ[G]$-module.

Indeed, let $V$ be the $\bbZ[G]$-submodule of $\omega(G/H)$ generated
by these elements.  Lemma~\ref{lem.subgroup} and condition (i)
tell us that $V$ contains $\overline{g} - \overline{1}$ for 
every $g \in G$. Translating these elements by $G$, we see 
that $V$ contains $\overline{a} - \overline{b}$ for every $a, b \in G$. 
Hence, $V = \omega(G/H)$, as claimed.

For $i = 1, \dots, s$, let 
\[ S_i := \{ g \in G \, | \, g \cdot (\overline{g_i} - 
\overline{1}) = \overline{g_i} - \overline{1} \} \] 
be the stabilizer of $\overline{g_i} - \overline{1}$ in $G$.
We may assume here that $g_i$ is not in $H$, otherwise it could be removed since it is not needed to generate $G$ over $H$.
Then clearly $g \in S_i$ iff $\overline{g g_i} = \overline{g_i}$ and
$\overline{g} = \overline{1}$. From this one easily sees that
$S_i = H \cap H^{g_i}$. Thus we have an exact sequence
\[ 0 \to M \to \oplus_{i = 1}^s \bbZ [G/S_i] 
\xrightarrow{\phi} \omega(G/H) \to 0 \]
where $\phi$ sends a generator of $\bbZ[G/S_i]$ to $\overline{g_i} - \overline{1} \in \omega(G/H)$.
By Theorem~\ref{thm.lrrs} it remains to show that $G$ acts faithfully on $M$.

By Lemma~\ref{lem2.1'} $G$ fails to act faithfully on $M$ if and only 
if $r = 1$ and $S_1 = H = H^{g_1}$. But this 
possibility is ruled out by (ii). Indeed, assume that
$s = 1$ and $S_1 = H = H^{g_1}$. Then $G = \langle g_1, H \rangle$ 
and $H = H^{g_1}$. Hence, $H$ is normal in $G$. 
Condition~\eqref{e.H} then tells us that $H = \{ 1 \}$. Moreover,
in this case
$G = \langle g_1, H \rangle = \langle g_1 \rangle$ is cyclic, 
contradicting (ii).
\end{proof}

\section{Proof Theorem~\ref{thm.csa}} 

As we saw above, it suffices to prove Theorem~\ref{thm.csa'}.

Let $t_1, \dots, t_r \in G/N$ be a set of generators for $G/N$.
Choose $g_1, \dots, g_r \in G$ representing $t_1, \dots, t_r$.
and let $H' := \langle H, H^{g_1}, \ldots, H^{g_r} \rangle$.
Since $H \le N$ and $N$ is normal in $G$, $H' \le N$.
The group $H'$ depends on the choice of $g_1, \dots, g_r \in G$,
so that $g_i N = t_i$. Fix $t_1, \dots, t_r$ and choose 
$g_1, \dots, g_r \in G$ representing them, so that $H'$ 
has the largest possible order or equivalently the smallest possible
index in $N$. Denote this minimal possible value of $[N: H']$
by $m$. In particular
\begin{equation} \label{e.N}
m = [N : H'] \le [N : (H^{g_i g} \cdot H)] 
\end{equation}
for any $i = 1, \dots, r$ and any $g \in N$. Here   
$[N : (H^{g_i g} \cdot H)] = \dfrac{|N|}{|H^{g_i g} \cdot H|}$,
as in Remark~\ref{rem.H}.  

Choose a set of representatives
$1 = n_1, n_2, \ldots, n_m \in N$ for the distinct left cosets of
$H'$ in $N$. We claim that the elements
\[ \{ g_i n_j \, | \, i = 1, \dots, r; \; j = 1, \dots, m \} \]
generate $G$ over $H$. 
Indeed, let $G_0$ be the subgroup of $G$ generated by these elements 
and $H$. Since $n_1 = 1$, $G_0$ contains $g_1, \dots, g_r$. 
Hence, $G_0$ contains $H'$. Moreover, $G_0$ contains 
$n_j = g_1^{-1}(g_1 n_j) $ for every $j$; hence, $G_0$
contains all of $N$. Finally, since 
$t_1 = g_1 N, \dots, t_r = g_r N$ generate $G/N$, 
we conclude that $G_0$ contains all of $G$. This proves the claim.    

We now apply Theorem~\ref{thm.H} to the elements $\{ g_i n_j \}$.
Substituting 
\[ \text{$[G: H] \cdot [H : (H \cdot H^{g_i n_j})]$
for
$[G: (H \cap H^{g_i n_j})]$,} \]
as in Remark~\ref{rem.H}, we obtain
\begin{align*}
\ed(A) & \le & \sum_{i = 1}^r \sum_{j = 1}^m 
[G: (H \cap H^{g_i n_j})]  - [G:H] + 1 \\  
       & = & 
[G : H] \cdot 
\sum_{i = 1}^r \sum_{j = 1}^m 
 [(H \cdot H^{g_i n_j}) : H]  - [G:H] + 1 \\  
       & = & 
[G : H] \cdot 
\sum_{i = 1}^r \sum_{j = 1}^m 
\frac{[N : H]}{[N : (H \cdot H^{g_i n_j})]} - [G:H] + 1 \\  
       & \le & \text{(by~\eqref{e.N})} \quad 
[G : H] \cdot 
\sum_{i = 1}^r \sum_{j = 1}^m 
\frac{[N : H]}{m} - [G:H] + 1 \\  
       & = & 
r [G: H] \cdot [N : H] - [G:H] + 1 
\end{align*}
as desired.
This completes the proof of Theorem~\ref{thm.csa'} and 
thus of Theorem~\ref{thm.csa}.
\qed

\section*{Acknowledgments} The authors are grateful to 
B. A. Sethuraman, D. J. Saltman, and B. Totaro for helpful comments.

\end{document}